\documentclass[12pt, reqno]{amsart}
\usepackage{amsmath, amsthm, amscd, amsfonts, amssymb}
\usepackage{amsmath, amssymb, mathrsfs}
\theoremstyle{plain}
\newtheorem{theorem}{Theorem}[section]
\newtheorem*{mainclaim}{Main Theorem}
\newtheorem{lemma}[theorem]{Lemma}
\newtheorem{proposition}[theorem]{Proposition}
\newtheorem{corollary}[theorem]{Corollary}
\theoremstyle{definition}
\newtheorem{definition}[theorem]{Definition}
\theoremstyle{remark}
\newtheorem{remark}[theorem]{Remark}

\newcommand{\R}{\mathbb R}
\newcommand{\N}{\mathbb N}
\newcommand{\B}[1]{B_{#1}}
\newcommand{\Sph}[1]{S_{#1}}
\newcommand{\Q}[1]{\mathcal Q(#1)}
\newcommand{\DK}{\mathcal D(K)}
\newcommand{\norm}[1]{\left\lVert #1\right\rVert}
\newcommand{\spanR}{\operatorname{span}_{\R}}
\newcommand{\affR}{\operatorname{aff}_{\R}}
\DeclareMathOperator{\Supp}{Supp}
\DeclareMathOperator{\Id}{Id}

\begin{document}
\setcounter{page}{1}

\title[The SMP for Closed Convex Body in Normed Spaces]{The Strong Mankiewicz Property for Closed Convex Body in Normed Spaces}

\author{Ruidong Wang}
\address{School of Mathematical Sciences, Tianjin University of Technology, Tianjin 300384, China}
\email{wangruidong@tjut.edu.cn}

\author{Jianglong Luo}
\address{School of Mathematical Sciences, Tianjin University of Technology, Tianjin 300384, China}

\email{2100056943@qq.com}

\author{Li Tang}
\address{School of Mathematical Sciences, Tianjin University of Technology, Tianjin 300384, China}

\email{tangli@stud.tjut.edu.cn}

\let\thefootnote\relax\footnote{ The study is supported by the Natural Science Foundation of China (Grant No.12271402) and Natural Science Foundation of Tianjin Municipal Science and Technology Commission (Grant No. 23JCYBJC00560, 22JCYBJC00420 )}



\date{Received: xxxxxx; Revised: yyyyyy; Accepted: zzzzzz.
}

\begin{abstract}
Mori and Ozawa~\cite{MO} raised a fundamental open question: whether every convex body in a normed space enjoys the strong Mankiewicz property (SMP). The SMP states that every surjective isometry from the convex body onto a convex subset of an arbitrary real normed space is necessarily real-affine, with no interior assumption imposed on the target convex set. In the present work, we settle this problem affirmatively.
\end{abstract}

\maketitle

\section{Introduction}

The Mazur-Ulam theorem characterizes surjective isometries between real
normed spaces as affine mappings. Mankiewicz \cite{Mankiewicz} extended this conclusion to
surjective isometries between convex bodies, where a convex body is understood to be
a convex set with nonempty norm-interior.

A convex subset $D$ of a normed
space is said to have  the \emph{strong Mankiewicz property} if every surjective
isometry from $D$ onto a convex subset of an arbitrary normed space is real-affine.
The SMP was introduced by Mori and Ozawa~\cite{MO}. Since there exist convex sets with empty interior failing this property, as demonstrated by a counterexample in~\cite{Sc}, they raised the question whether every convex body in a normed space has the SMP. Mori and Ozawa   verified this property
for convex bodies in a Banach space under the assumption that the closed convex hull of the
extreme points of its closed unit ball has nonempty interior.

 As shown in \cite[Lemma~4]{MO}, if the closed unit ball $B_X$ of a Banach space $X$ possesses the SMP, then all convex bodies in $X$ inherit the property.
Consequently, it suffices to verify the SMP solely for the closed unit ball $B_X$.

All spaces considered below are real. For complex normed spaces, all statements are interpreted on their underlying real normed spaces. For a normed space $X$,  we write
\[
  B_{X}=\{x\in X:\|x\|\leq1\},
  \qquad
  S_{X}=\{x\in X:\|x\|=1\}.
\]
Closures, interiors  and boundaries are taken with respect to the norm topology.
For a subset $D\subset X$, we write
$\overline D^{\|\cdot\|}$, $\operatorname{int}_X D$, and $\partial_XD$ when the ambient space $X$ needs to be specified. The notations $\spanR D$ and $\affR D$ stand for the algebraic real linear hull and the algebraic real affine hull of $D$, respectively.

In this paper, we establish the following main theorem, which settles  Mori and Ozawa's problem affirmatively: every convex body in a normed space has the SMP.

\begin{mainclaim}\label{claim:main}
Let $X$ and $Y$ be real normed spaces, let $C\subset Y$ be a nonempty
convex set, and let
\[
  f:B_{X}\longrightarrow C
\]
be a surjective isometry. Then $f$ is affine. More precisely, there exists a
unique real-linear isometric embedding $T:X\to Y$ such that
\begin{equation}\label{eq:extension}
  f(x)=f(0)+Tx\qquad(x\in B_{X}).
\end{equation}
Moreover,
\begin{equation}\label{eq:affine-hull}
  \affR C=f(0)+T(X).
\end{equation}
If $X$ is Banach, then $T(X)$ is closed in $Y$ and
\[
  C-f(0)=B_{T(X)}
\]
with respect to the norm inherited from $Y$.
\end{mainclaim}

Briefly, our approach works as follows. Given a surjective isometry defined on the closed unit ball, we construct a scaled sequence map into the quotient space \(\mathcal Q(E)=\ell_\infty(\mathbb N;E)/c_0(\mathbb N;E)\). Using Figiel's theorem, we deduce linearity for this quotient-space map. Combining this with Bishop-Phelps support-point argument, we show that the target convex set has nonempty interior, so Mankiewicz's theorem can be applied. Finally we handle incomplete normed spaces via metric completion.

\section{Preliminaries}

We recall several classical results that will be used in the sequel; proofs are omitted. The following  theorem is originally due to Figiel \cite{Figiel}. A proof can be found in
Godefroy \cite[Theorem~4 and the following Remark~1]{Godefroy}.

\begin{theorem}[Figiel]\label{thm:figiel}
Let $E$ and $F$ be real Banach spaces. If $J:E\to F$ is an isometric
embedding satisfying $J(0)=0$, then there exists a continuous real-linear operator
\[
  A:\overline{\spanR J(E)}^{\|\cdot\|}\longrightarrow E,
  \qquad \norm{A}\leq1,
\]
such that $A\circ J=\Id_E$.
\end{theorem}

\begin{corollary}\label{cor:whole-space}
Under the assumptions of Theorem~\ref{thm:figiel}, if $J(E)$ is convex,
then $J$ is real-linear.
\end{corollary}
\begin{proof}
Let \(u,v\in E\) and \(\lambda\in[0,1]\). Since \(J(E)\) is convex,
$$(1-\lambda)J(u)+\lambda J(v)\in J(E).$$
By definition of \(J(E)\), there exists some \(w\in E\) such that
$$J(w)=(1-\lambda)J(u)+\lambda J(v).$$
Apply the operator $A$ from Figiel's theorem to both sides. Since $A$ is real-linear,
\begin{eqnarray*}
w&=&\operatorname{Id}_E(w)=A(J(w))\\
&=&A\big((1-\lambda)J(u)+\lambda J(v)\big)\\
&=&(1-\lambda)A(J(u))+\lambda A(J(v))\\
&=&(1-\lambda)u+\lambda v.
\end{eqnarray*}

Thus
$$J\big((1-\lambda)u+\lambda v\big)
=(1-\lambda)J(u)+\lambda J(v),\quad \forall u,v\in E,\;\lambda\in[0,1].$$
This shows that $J$ is affine. Together with the assumption \(J(0)=0\), an affine map vanishing at origin is necessarily real-linear.
Therefore, $J$ is real-linear.
\end{proof}

\begin{definition}\label{def:support}
Let $D$ be a nonempty subset of a real Banach space $X$. A point $q\in D$
is called a \emph{support point} of $D$ in $X$ if there exists a nonzero
$\psi\in X^*$ such that
\[
  \psi(y)\leq\psi(q)\qquad(y\in D).
\]
We denote the collection of all such points by $\Supp_X(D)$.
\end{definition}

\begin{theorem}[Bishop--Phelps support point theorem \cite{BP}]\label{thm:BP}
Let $D$ be a nonempty closed convex subset of a real Banach space $X$.
Then
\begin{equation}\label{eq:BP}
  \partial_XD\subset\overline{\Supp_X(D)}^{\|\cdot\|}.
\end{equation}
In particular, the set of support points is norm dense in the norm boundary of $D$.
\end{theorem}

\begin{theorem}[Mankiewicz \cite{Mankiewicz}]\label{thm:mankiewicz}
Let $E$ and $F$ be real normed spaces. Let $D\subset E$ and $K\subset F$
be closed convex sets with nonempty norm interiors. Every surjective
isometry $h:D\to K$ extends uniquely to a surjective affine isometry
from $E$ onto $F$.
\end{theorem}
\begin{remark}\label{rem:complete-only}
Let $X,Y$ be real normed spaces, let $C\subset Y$ be convex, and let
$f:\B{X}\to C$ be a surjective isometry. By the standard extension of isometries to metric completions, and since the closure of a convex set remains convex, it suffices to first treat the case of Banach spaces. In that
case, we translate the range by $-f(0)$ and retain the same notation for the
translated map and set. We then have
\begin{equation}\label{eq:normalized}
  f(0)=0,\qquad 0\in C,\qquad
  \norm{f(x)}=\norm{x}\quad(x\in\B{X}).
\end{equation}
Since $\B{X}$ is complete and $f$ is a surjective isometry, $C$ is
complete and therefore closed in $Y$. Set
\begin{equation}\label{eq:ambient}
  Z=\overline{\spanR C}^{\|\cdot\|}.
\end{equation}
Then $Z$ is Banach, $C$ is closed in $Z$, and the closed real affine
hull of $C$ equals $Z$.
\end{remark}

For a real Banach space $E$, define
\begin{align*}
  \ell_\infty(\N;E)
  &=\{(u_n)_{n\geq1}:\sup_n\norm{u_n}<\infty\},\\
  c_0(\N;E)
  &=\{(u_n)_{n\geq1}:\lim\limits_{n\to \infty}\norm{u_n}=0\},\\
  \Q{E}&=\ell_\infty(\N;E)/c_0(\N;E).
\end{align*}
The space $\ell_\infty(\N;E)$ is Banach and $c_0(\N;E)$ is a closed
linear subspace, so $\Q{E}$ is a real Banach space. The equivalence class of a
bounded sequence $(u_n)$ is denoted by $[u_n]$. For $v\in E$, let
\[
  \iota_E(v)=[v,v,\ldots]
\]
denote the constant-sequence embedding.

We use the standard quotient-norm formula
\begin{equation}\label{eq:quotient-norm}
  \norm{[u_n]}_{\Q{E}}=\limsup_{n\to\infty}\norm{u_n}.
\end{equation}
Consequently, $[u_n]=[v_n]$ if and only if $\norm{u_n-v_n}\to0$.
The map $\iota_E:E\to\Q{E}$ is a linear isometric embedding.

\section{Main results}

\begin{proposition}\label{prop:rescaling}
Let $X,Z$ be real Banach spaces, let $K\subset Z$ be convex, and let
$g:\B{X}\to K$ be a surjective isometry with $g(0)=0$.
For $n\geq1$, define a map on all of $X$ by
\[
  H_n(x)=
  \begin{cases}
    n g(x/n),&\norm{x}\leq n,\\
    0,&\norm{x}>n.
  \end{cases}
\]
Then
\begin{equation}\label{eq:Phi}
  \Phi:\Q{X}\longrightarrow\Q{Z},
  \qquad \Phi([x_n])=[H_n(x_n)]
\end{equation}
is a well-defined real-linear isometric embedding. Its range is
\begin{equation}\label{eq:DK}
  \DK=
  \bigl\{[y_n]:(y_n)\in\ell_\infty(\N;Z),\
                \exists k\in \mathbb{N}\text{ such that }\  y_n\in nK\  \text{ for all } n\ge k\bigr\}.
\end{equation}
In particular, $\DK$ is a closed real-linear subspace of $\Q{Z}$.
\end{proposition}

\begin{proof}

Since $g$ fixes the origin, we have $\norm{g(x)}=\norm{x}$ for $x\in\B{X}$.
It follows that
\[
  \norm{H_n(x)}\leq\norm{x},\qquad \forall x\in X.
\]

If $(x_n)$ and $(x'_n)$ are bounded sequences, then for all  sufficiently large $n$
both $\norm{x_n}$ and $\norm{x'_n}$ are at most $n$, and hence
\begin{equation}\label{eq:tail-distance}
  \norm{H_n(x_n)-H_n(x'_n)}=\norm{x_n-x'_n}.
\end{equation}
Therefore, input sequences differing by a null sequence yield output sequences
differing by a null sequence. This proves well-definedness.
By \eqref{eq:quotient-norm},
\begin{align*}
  \norm{\Phi([x_n])-\Phi([x'_n])}_{\Q{Z}}
  &=\limsup_n\norm{H_n(x_n)-H_n(x'_n)}\\
  &=\limsup_n\norm{x_n-x'_n}\\
  &=\norm{[x_n]-[x'_n]}_{\Q{X}}.
\end{align*}
Thus $\Phi$ is an isometry and $\Phi(0)=0$.

Since $0\in K$, we have $H_n(x_n)\in nK$ for every $n$, so
$\Phi(\Q{X})\subset\DK$. Conversely, let $[y_n]\in\DK$ and choose
a bounded representative with $y_n\in nK$ for $n\geq N$. Define
\begin{equation}\label{eq:inverse-sequence}
  x_n=n g^{-1}(y_n/n)\quad(n\geq N),
  \qquad x_n=0\quad(n<N).
\end{equation}
The inverse is well-defined because $y_n/n\in K$. As $g^{-1}$ preserves
distance from zero,
\begin{equation}\label{eq:bounded-inverse}
  \norm{x_n}=n\norm{y_n/n}=\norm{y_n}\qquad(n\geq N).
\end{equation}
Thus $(x_n)$ is bounded. Also $x_n/n\in\B{X}$ and
$H_n(x_n)=y_n$ for $n\geq N$. Finite initial differences vanish in
the quotient, proving $\Phi([x_n])=[y_n]$ and hence
$\Phi(\Q{X})=\DK$.

Let $[y_n],[z_n]\in\DK$ and $t\in[0,1]$. Choose bounded
representatives that eventually lie in $nK$. For sufficiently large
$n$, convexity of $nK$ gives
\[
  (1-t)y_n+tz_n\in nK.
\]
The resulting sequence is bounded since
\[
  \norm{(1-t)y_n+tz_n}\leq(1-t)\norm{y_n}+t\norm{z_n}.
\]
Therefore $\DK$ is convex.

The spaces $\Q{X}$ and $\Q{Z}$ are real Banach spaces. We have shown
that $\Phi$ is an isometric embedding fixing 0 and whose range is convex. Corollary~\ref{cor:whole-space} therefore implies that $\Phi$ is
real-linear. Its range is consequently a linear subspace. The range
is complete, being isometric to $\Q{X}$, and so it is closed in
$\Q{Z}$.
\end{proof}

\begin{remark}\label{rem:quotient-only}
The linearity in Proposition~\ref{prop:rescaling} concerns $\Phi$ on
the entire sequence quotient. No linearity of the original ball map
$g$ has been assumed or inferred at this stage.
\end{remark}

\begin{lemma}\label{lem:center}
Under the hypotheses of Proposition~\ref{prop:rescaling}, if $\psi\in Z^*$
satisfies $\psi(z)\leq0$ for every $z\in K$, then $\psi$ vanishes on $K$.
\end{lemma}

\begin{proof}
Take any $v\in K$. By convexity of $K$ and the fact that $0\in K$, we have $v/n\in K$ for all integers
$n\geq1$. Hence
$$\iota_Z(v)\in\DK.$$

By Proposition~\ref{prop:rescaling}, $\DK$ is  a linear subspace,
so $$-\iota_Z(v)=\iota_Z(-v) \in \DK.$$

By the definition of $\DK$, there exists a bounded sequence $(y_n)$ and some $k\in \mathbb{N}$ such that
$y_n\in nK$ for all $n\ge k$, and
\[
  [y_n]=\iota_Z(-v).
\]

Therefore,
\begin{equation}\label{eq:negative-approximation}
  \lim\limits_{n\to\infty}\norm{y_n+v}=0.
\end{equation}

For each $n\ge k$, write $y_n=nk_n$ where $k_n\in K$. Then
$\psi(y_n)=n\psi(k_n)\leq0$. By continuity of $\psi$ and
\eqref{eq:negative-approximation}, we obatin $\psi(-v)\leq0$, or equivalently $\psi(v)\geq0$.
The hypothesis yields the reverse inequality, so $\psi(v)=0$.
Since $v\in K$ was arbitrary, $\psi$ vanishes on $K$.
\end{proof}

\begin{lemma}\label{lem:local-ball}
Let $X,Z$ be real normed spaces, let $C\subset Z$ be convex, and let
$f:\B{X}\to C$ be a surjective isometry. Suppose $p\in X$ and $r>0$
satisfy $p+r\B{X}\subset\B{X}$, and set $q=f(p)$. Then the map
\[
  g(x)=\frac{f(p+rx)-q}{r},\qquad \forall x\in\B{X}
\]
is an   isometry. Moreover, $g(0)=0$, $g(B_X)$ is convex, and
\begin{equation}\label{eq:local-image}
  g(B_X)=K=\frac{C-q}{r}\cap\B{Z}.
\end{equation}
\end{lemma}

\begin{proof}
The definition together with the isometric property of $f$ yields $g(0)=0$ and
\[
  \norm{g(x)-g(x')}=\norm{x-x'}\qquad \forall x,x'\in\B{X}.
\]
For $x\in\B{X}$, we have $f(p+rx)\in C$ and
$\norm{g(x)}=\norm{x}\leq1$, which implies $g(\B{X})\subset K$.

Conversely, take $v\in K$. Then $q+rv\in C$. Let
$u=f^{-1}(q+rv)\in\B{X}$. Since
\[
  \norm{u-p}=\norm{f(u)-q}=r\norm{v}\leq r,
\]
we obtain  $x=(u-p)/r\in\B{X}$ and $g(x)=v$. This establishes the equality
in \eqref{eq:local-image}. The set on the right-hand side is convex, as it is the intersection of two convex sets.
\(\square\)
\end{proof}

\begin{proposition}\label{prop:support}
Let $X,Z$ be real Banach spaces, let $C\subset Z$ be convex, and let
$f:\B{X}\to C$ be a surjective isometry satisfying
\begin{equation}\label{eq:full-span}
  f(0)=0,\qquad \overline{\spanR C}^{\|\cdot\|}=Z.
\end{equation}
Then
\begin{equation}\label{eq:support-sphere}
  \Supp_Z(C)\subset f(\Sph{X}).
\end{equation}
\end{proposition}

\begin{proof}
Take $p\in X$ with $\norm{p}<1$, set $q=f(p)$  and define
\[
  r=\frac{1-\norm{p}}2>0.
\]
Then $p+r\B{X}\subset\B{X}$. Apply Lemma~\ref{lem:local-ball} to
obtain a surjective isometry
$$g(x)=\frac{f(p+rx)-q}{r}:\B{X}\to K=\frac{C-q}{r}\cap\B{Z},$$ which fixes the origin.

Suppose $q\in\Supp_Z(C)$. Then there exists a nonzero $\psi\in Z^*$ such that
\begin{equation}\label{eq:supporting}
  \psi(y)\leq\psi(q),\qquad\forall y\in C.
\end{equation}
For any $k\in K$, we have $q+rk\in C$, so \eqref{eq:supporting} yields
$\psi(k)\leq0$. Applying Lemma~\ref{lem:center}  to $g$, we now get
\begin{equation}\label{eq:vanishing-local}
  \psi(k)=0,\qquad\forall k\in K.
\end{equation}

Fix $y\in C$. Choose $s\in(0,1]$ such that $s\norm{y-q}\leq r$.
Convexity of $C$ implies $q+s(y-q)\in C$, and hence
\[
  \frac{s(y-q)}r\in K.
\]
Equation~\eqref{eq:vanishing-local} then gives $\psi(y)=\psi(q)$.
Thus $\psi$ is constant on $C$. Since $0\in C$, this constant value is zero.
Linearity forces $\psi$ to vanish on $\spanR C$, and together with
\eqref{eq:full-span}, continuity implies $\psi=0$ on $Z$, a contradiction.

We have therefore shown that $f(p)$ cannot be a support point whenever
$\norm{p}<1$. Surjectivity of $f$ then yields
$$ \Supp_Z(C)\subset f(\Sph{X}).$$
This completes the proof.

\end{proof}

\begin{proposition}\label{prop:interior}
Under the hypotheses of Proposition~\ref{prop:support}, if $C$ is
closed in $Z$, then
\begin{equation}\label{eq:interior}
  f\bigl(\operatorname{int}_X\B{X}\bigr)
  \subset\operatorname{int}_Z C.
\end{equation}
In particular, $0\in\operatorname{int}_Z C$.
\end{proposition}

\begin{proof}
The isometry $f$ is a homeomorphism from $\B{X}$ onto $C$.
Since $\Sph{X}$ is closed in $\B{X}$, its image $f(S_X)$ is closed in $C$.
As $C$ itself is closed in $Z$, $f(\Sph{X})$ is closed in $Z$.

By the Bishop-Phelps theorem and together with Proposition~\ref{prop:support},
\begin{equation}\label{eq:boundary-sphere}
  \partial_Z C
  \subset\overline{\Supp_Z(C)}^{\|\cdot\|}
  \subset f(\Sph{X}).
\end{equation}
If $\norm{p}<1$, injectivity of $f$ implies $f(p)\notin f(\Sph{X})$, so
$f(p)\notin\partial_Z C$. For a closed set,
\[
  C\setminus\partial_Z C=\operatorname{int}_Z C.
\]
Since $f(p)\in C$, we conclude that $f(p)\in\operatorname{int}_Z C$.
This establishes \eqref{eq:interior}, and substituting $p=0$ yields the final assertion.
\end{proof}

\begin{mainclaim}
Let $X$ and $Y$ be real normed spaces, let $C\subset Y$ be a nonempty
convex set, and let
\[
  f:B_{X}\longrightarrow C
\]
be a surjective isometry. Then $f$ is affine. More precisely, there exists a
unique real-linear isometric embedding $T:X\to Y$ such that
\begin{equation}\label{eq:extension}
  f(x)=f(0)+Tx\qquad(x\in B_{X}).
\end{equation}
Moreover,
\begin{equation}\label{eq:affine-hull}
  \affR C=f(0)+T(X).
\end{equation}
If $X$ is Banach, then $T(X)$ is closed in $Y$ and
\[
  C-f(0)=B_{T(X)}
\]
with respect to the norm inherited from $Y$.
\end{mainclaim}

\begin{proof}
If $X=\{0\}$, then $C$ is a singleton, and all claims follow immediately. Assume $X\neq\{0\}$. By Remark~\ref{rem:complete-only},
 we may assume that $X$ and $Y$ are Banach  and
$f(0)=0$. Set
$$Z=\overline{\spanR C}^{\|\cdot\|}.$$
 Then $Z$ is Banach and $C$ is closed in $Z$.

Propositions~\ref{prop:support} and \ref{prop:interior} yield
$0\in\operatorname{int}_Z C$. Thus $\B{X}$ and $C$ are closed convex
sets with nonempty interiors in $X$ and $Z$, respectively.
Mankiewicz's theorem extends $f$ uniquely to a surjective affine
isometry $T:X\to Z$. Since $T(0)=0$, this extension is real-linear.
Reversing the translation recovers formula \eqref{eq:extension}.

Any two linear extensions that agree on $\B{X}$ coincide on all   $x\in X$ by scaling. Surjectivity of $f$ implies
$C-f(0)=T(\B{X})$. Since $0\in C-f(0)$ and $\spanR\B{X}=X$,
\begin{align*}
  \affR C
  &=f(0)+\spanR\bigl(C-f(0)\bigr)\\
  &=f(0)+T(X),
\end{align*}
which establishes \eqref{eq:affine-hull}.

If $X$ is Banach, then $T(X)$ is complete under the inherited norm, hence closed in $Y$ even if $Y$ itself is not complete. Since $T$ maps $X$
onto $T(X)$ isometrically, $T(\B{X})=\B{T(X)}$. This proves the final assertion of the main theorem.
\end{proof}
\begin{remark}
For complex normed spaces, the same conclusion holds by passing to the underlying real normed space.
\end{remark}

Mori and Ozawa \cite[Lemma~4]{MO} showed that if $\B{X}$ possesses the strong
Mankiewicz property, then every convex body in $X$ also has this property.
This yields the following consequence of the Main Theorem.

\begin{corollary}\label{cor:all-bodies}
Every convex body in any real
or complex normed space has the strong Mankiewicz property, where
affinity is interpreted over the real scalars.
\end{corollary}

\subsection*{Acknowledgements}
The authors are grateful to the referees for their constructive comments and suggestions.


\begin{thebibliography}{99}
\bibitem{Figiel}
T.~Figiel,
On nonlinear isometric embeddings of normed linear spaces,
\emph{Bull. Acad. Polon. Sci. S\'er. Sci. Math. Astronom. Phys.}
\textbf{16} (1968), 185--188.

\bibitem{Godefroy}
G.~Godefroy,
Linearization of isometric embeddings between Banach spaces:
An elementary approach,
\emph{Operators and Matrices} \textbf{6} (2012), no.~2, 339--345.


\bibitem{BP}
E.~Bishop and R.~R.~Phelps,
The support functionals of a convex set,
in \emph{Convexity}, Proceedings of Symposia in Pure Mathematics,
vol.~7, American Mathematical Society, 1963, pp.~27--35.

\bibitem{Mankiewicz}
P.~Mankiewicz,
On extension of isometries in normed linear spaces,
\emph{Bull. Acad. Polon. Sci. S\'er. Sci. Math. Astronom. Phys.}
\textbf{20} (1972), 367--371.

\bibitem{MO}
M.~Mori and N.~Ozawa,
Mankiewicz's theorem and the Mazur--Ulam property for $C^*$-algebras,
\emph{Studia Mathematica} \textbf{250} (2020), no.~3, 265--281.

\bibitem{Sc}
G. Schechtman,
Generalizing the Mazur-Ulam theorem to convex sets with empty interior,
\emph{MathOverflow}, https://mathoverflow.net/questions/225577.
\end{thebibliography}
\end{document}